\documentclass[a4paper,12pt,reqno]{amsart}
\usepackage{amssymb,amsmath,amsfonts, amscd}
\usepackage{mathrsfs,latexsym}
\usepackage[all,ps,cmtip]{xy}
\usepackage[usenames,dvipsnames,svgnames,table]{xcolor}
\usepackage{extarrows}
\usepackage{graphicx}

\title{On the pluricanonical map and the canonical volume of projective 4-folds of general type}
\author{Jianshi Yan}

\address{\rm Department of Mathematics, Northeastern University, Shenyang 110819, China}
\email{yanjs@mail.neu.edu.cn}

\newcommand{\bQ}{{\mathbb Q}}
\newcommand{\bP}{{\mathbb P}}
\newcommand{\roundup}[1]{\lceil{#1}\rceil}

\newcommand\Vol{\text{\rm Vol}}

\newcommand{\lsgeq}{\succcurlyeq}
\newcommand{\lsleq}{\preccurlyeq}

\newcommand{\simQ}{\sim_{\mathbb{Q}}}

\newtheorem{thm}{Theorem}[section]
\newtheorem{lem}[thm]{Lemma}

\theoremstyle{definition}

\newtheorem{rem}[thm]{Remark}
\theoremstyle{remark}

\newtheorem{proof of thm}{\bf Proof of Theorem \ref{modified lemma for 3-fold}}
\begin{document}
\begin{abstract}

For nonsingular projective 4-folds V of general type with plurigenus $P_{m_0}(V) \geq 2$ for some positive integer $m_0$, we show that $\varphi_{m}$ is birational onto its image for all integers $m \geq 76m_0+77$ and the canonical volume $\Vol(V)$ has the lower bound $\frac{1}{1680m_0(m_0+1)^3}$. This improves earlier results of Meng Chen.

\end{abstract}
\maketitle

\pagestyle{myheadings}
\markboth{\hfill J. Yan\hfill}{\hfill On pluricanonical map and canonical volume of projective 4-folds of general type \hfill}
\numberwithin{equation}{section}

\section{\bf Introduction}

Understanding the behavior of pluricanonical maps and the lowest bound of canonical volumes of projective varieties has been a major question in birational geometry. A crucial theorem given by Hacon-McKernan \cite{H-M}, Takayama \cite{Ta} and Tsuji \cite{Tsuji} shows that, for any integer $n>0$, there are 
optimal constants $r_n \in \mathbb{Z}_{>0}$ and $v_n \in \mathbb{Q}_{>0}$
such that the pluricanonical map $\varphi_m$ := $\varphi_{|mK_V|}: V \dashrightarrow \mathbb{P}(H^0(V, \mathcal{O}_V(mK_V)))$ is birational onto its image for all $m \geq r_n$ and the canonical volume $\Vol(V):=\lim\limits_{m\to \infty}\frac{n!h^0(V, mK_V)}{m^{n}} \geq v_n$ for all nonsingular projective $n$-folds $V$ of general type.
Therefore, it is interesting to know the values of the numbers $r_n$ and $v_n$.

For $n=1$, it is known that $r_1=3$. For $n=2$, Bombieri showed in \cite{Bom} that $r_2=5$. For 3-folds, Iano-Fletcher \cite{IF}, Chen-Chen \cite{EXP1, EXP2, EXP3} and Chen \cite{Delta18} have shown that $27 \leq r_3 \leq 57$. As the classification of terminal singularities of 4-folds is still unclear, there is no effective Riemann-Roch formula for 4-folds to calculate $\chi(mK)$ explicitly. Consequently, it is still unknown when the plurigenus $P_m(V):= \dim H^0(V, \mathcal{O}_V(mK_V)) \geq 2$ holds for a 4-fold $V$ of general type and very little is known about $r_4$. For a nonsingular projective 3-fold $V$ of general type, suppose $P_{m_0}(V) \geq 2$, Koll\'ar  first proved in \cite[Corollary 4.8]{Kol} that the $(11m_0+5)$-canonical map is birational. Then Chen \cite{Q-div} improved the method of Koll\'ar and showed that the $m$-canonical map of $V$ is birational onto its image for all $m \geq 5m_0+6$. For a nonsingular projective 4-fold $V$ of general type, granted that $P_{m_0}(V) \geq 2$ for some positive integer $m_0$, Meng Chen proved in \cite{Chen10} that $\varphi_{m}$ is birational onto its image for all integers $m \geq 151m_0+77$.

As for the lower bound of the volume, given a projective variety $V$ of general type, by the minimal model program(see \cite{BCHM, KMM,  K-M,  Siu}), one can always find a minimal model $X$ birational to $V$. For the canonical volume, one knows that $\Vol(V)=K_X^{\dim X}$. It is known that $v_1=2$ and $v_2=1$.
As in dimension three or higher, a minimal model may have singularities, thus $v_n$ is only a positive rational number.
Chen-Chen (\cite{EXP1, EXP2, EXP3}) and Iano-Fletcher \cite{IF} showed that $\frac{1}{1680} \leq v_3 \leq \frac{1}{420}$. When $n=4$, some partial results are known:

$\bullet$ In 1992, Kobayashi \cite{Kob} proved that for a minimal projective 4-fold $Y$ of general type with $p_g(Y) \geq 2$ and $\dim\varphi_{|K_Y|}(Y)=4$, $\mathrm{Vol}(Y) \geq 2p_g(Y)-8$.

$\bullet$ In 2007, Chen \cite{MA} proved that for a minimal projective 4-fold $Y$ of general type with only canonical singularities and $p_g(Y) \geq 2$, if $Y$ is not canonically fibered by either 3-folds of geometric genus 1 or any irrational pencil of 3-folds, then $\mathrm{Vol}(Y) \geq \frac{1}{81}$.

$\bullet$ In 2021, Yan \cite{Y} proved that for a nonsingular projective 4-fold $V$ of general type with $p_g(V) \geq 2$, $\mathrm{Vol}(V) \geq \frac{1}{480}$.

$\bullet$ In 2020, Chen, Jiang and Li \cite{CJL} proved that for a minimal projective 4-fold $Y$ of general type with $p_g(Y) \geq 2$, when $\dim \overline{\varphi_{|K_Y|}(Y)}=3$ (or $\dim \overline{\varphi_{|K_Y|}(Y)}=2$ respectively), $\mathrm{Vol}(Y) \geq \frac{2}{3}$ (or $\mathrm{Vol}(Y) \geq \frac{1}{6}$ respectively).

We can see that the partial results known for the volume all have the prerequisite that the geometric genus of the 4-fold is greater than or equal to $2$, which does not always hold. So we want to consider the more general case.

This paper aims to improve the result of Meng Chen in \cite{Chen10} on the birationality of the pluricanonical map $\varphi_m$ and to give a lower bound of the volume of nonsingular projective 4-folds of general type, in terms of $m_0$.

\begin{thm}\label{main theorem}
Let $V$ be a nonsingular projective $4$-fold of general type with $P_{m_0}(V) \geq 2$ for some positive integer $m_0$.
Then
\begin{itemize}
\item[(1)] $P_m(V)\ge 2$ for all $m \geq 38m_0+39$;
\item[(2)] $\varphi_{m}$ is birational onto its image for all integers $m \geq 76m_0+77$;
\item[(3)] $\Vol(V) \geq \frac{1}{1680m_0(m_0+1)^3}$.
\end{itemize}
\end{thm}

\section{\bf Preliminaries}\label{pre}

Throughout we work over an algebraically closed field $k$ of characteristic $0$. We use the following notations:
\begin{tabbing}
\= aaaaaaaaaaaaaaaaaaaaaa\= bbbbbbbbbbbbbbbbbbbbbbbbbbbbbbb \kill
\> $\sim$    \> linear equivalence\\
\> $\sim_{\mathbb{Q}}$  \> $\mathbb{Q}$-linear equivalence\\
\> $\equiv$  \> numerical equivalence\\
\> $|A| \lsgeq |B|$ or \> \\
\> equivalently $|B| \lsleq |A|$ \> $|A| \supseteq |B|+$ fixed effective divisors.
\end{tabbing}

\subsection{Convention}

\

For an arbitrary linear system $|D|$ of positive dimension on a normal projective variety $Z$, we may define {\it a generic irreducible element of $|D|$} in the following way. We have $|D|= \text{Mov}|D|+\text{Fix}|D|$, where $\text{Mov}|D|$ and $\text{Fix}|D|$ denote the moving part and the fixed part of $|D|$ respectively. Consider the rational map $\varphi_{|D|}=\varphi_{\text{Mov}|D|}$. We say that {\it $|D|$ is composed of a pencil} if $\dim \overline{\varphi_{|D|}(Z)}=1$; otherwise, {\it $|D|$ is not composed of a pencil}. {\it A generic irreducible element of $|D|$} is defined to be an irreducible component of a general member in $\text{Mov}|D|$ if $|D|$ is composed of a pencil or,  otherwise,  a general member of $\text{Mov}|D|$.

Keep the above settings. We say that $|D|$ can {\it distinguish different generic irreducible elements $X_1$ and $X_2$ of a linear system $|M|$} if neither $X_1$ nor $X_2$ is contained in $\text{Bs}|D|$,  and if $\overline{\varphi_{|D|}(X_1)} \nsubseteq \overline{\varphi_{|D|}(X_2)}, \overline{\varphi_{|D|}(X_2)} \nsubseteq \overline{\varphi_{|D|}(X_1)}$.

\subsection{Set up for the pluricanonical map $\varphi_{m_0,Y}$}\label{setup}
\

Let $V$ be a nonsingular projective 4-fold of general type. By the minimal model program (see, for instance \cite{BCHM, KMM,  K-M,  Siu}), one can always pick a minimal model $Y$ of $V$ with at worst $\mathbb{Q}$-factorial terminal singularities. As the plurigenus, the canonical volume and the behavior of the pluricanonical map are all birationally invariant in the category of normal varieties with canonical singularities, we may just study on $Y$ instead.

Denote by $K_Y$ the canonical divisor of $Y$. Let $m_0$ be a positive integer such that $P_{m_0}(Y)$  $= h^0(Y, \mathcal{O}_Y(m_0K_Y))\geq 2$.  Fix an effective divisor $K_{m_0} \sim m_0K_Y$. Guaranteed by Hironaka's theorem, we may take a series of blow-ups $\pi: Y' \longrightarrow Y$ such that:

(i) $Y'$ is nonsingular and projective;

(ii) the moving part of $|m_0K_{Y'}|$ is base point free so that $$g_{m_0}=\varphi_{m_0,Y}\circ\pi:Y' \longrightarrow \overline{\varphi_{m_0,Y}(Y)} \subseteq{\mathbb{P}^{P_{m_0}(Y)-1}}$$ is a non-trivial morphism;

(iii) the support of the union of $\pi^*(K_{m_0})$ and all those exceptional divisors of $\pi$ is of simple normal crossings.

Taking the Stein factorization of $g_{m_0}$, we get
$Y' \xlongrightarrow {f_{m_0}} \Gamma \xlongrightarrow s \overline{\varphi_{m_0,Y}(Y)}$ and the following commutative diagram:
\smallskip

\begin{picture}(50,80) \put(100,0){$Y$} \put(100,60){$Y'$}
\put(180,0){$\overline{\varphi_{m_0,Y}(Y)}$.} \put(200,60){$\Gamma$}
\put(115,65){\vector(1,0){80}}
\put(106,55){\vector(0,-1){41}} \put(203,55){\vector(0,-1){43}}
\put(114,58){\vector(2,-1){80}} \multiput(112,2.6)(5,0){13}{-}
\put(172,5){\vector(1,0){4}}
\put(145,70){$f_{m_0}$}
\put(208,30){$s$}
\put(92,30){$\pi$}
\put(132,-6){$\varphi_{m_0,Y}$}
\put(148,43){$g_{m_0}$}
\end{picture}
\bigskip

We may write
\begin{align}\label{eq: adjunction Y}
K_{Y'}=\pi^*(K_Y)+E_{\pi},
\end{align}
 where $E_{\pi}$ is a sum of distinct exceptional divisors with positive rational coefficients. Denote by $|M_m|$ the moving part of $|mK_{Y'}|$ for any positive integer $m$. We may write
 \begin{align*}
 m_0\pi^*(K_Y) \sim_{\mathbb{Q}} M_{m_0}+E_{m_0},
 \end{align*}
 where $E_{m_0}$ is an effective $\mathbb{Q}$-divisor.

If $\dim(\Gamma)=1$, we have $M_{m_0}\sim \sum\limits_{i=1}^{b} F_i\equiv bF$, where $F_i$ and $F$ are general fibers of $f_{m_0}$ and $b=\deg {f_{m_0}}_*\mathcal{O}_{Y'}(M_{m_0})\geq P_{m_0}(Y)-1$. More specifically, when $g(\Gamma)=0$, we say that $|M_{m_0}|$ {\it is composed of a rational pencil} and when $g(\Gamma)>0$, we say that $|M_{m_0}|$ {\it is composed of an irrational pencil}.

If $\dim(\Gamma)>1$,  by Bertini's theorem, we know that a general member of $|M_{m_0}|$ is nonsingular and irreducible.

Denote by $T'$ a generic irreducible element of $|M_{m_0}|$.  Set
$$\theta_{m_0}=\theta_{m_0, |M_{m_0}|}=
\begin{cases}
b, & \text{if}\ \dim(\Gamma)=1;\\
1, &\text{if}\ \dim(\Gamma)\geq 2.
\end{cases}$$
So we naturally get
\begin{align}\label{eq: mov mK}
m_0\pi^*(K_Y) \equiv \theta_{m_0}T'+E_{m_0}.
\end{align}

\subsection{Fixed notation}\label{notations}

\

Pick a generic irreducible element $T'$ of $|M_{m_0}|$. Let $t_1$ be a positive integer such that $P_{t_1}(T') \geq 2$. Let $\pi_T: T' \dashrightarrow T$ be the contraction map onto its minimal model $T$. Fix an effective divisor $K_{t_1} \sim t_1K_T$. Let $\nu: T'' \longrightarrow T'$ be the birational modification of $T'$ such that when $\vartheta:=\pi_T \circ \nu$, $\vartheta^*(K_{t_1}) \cup \{$Exceptional divisors of $\vartheta \}$ has simple normal crossing supports, Mov$|t_1K_{T''}|$ is base point free and $T$ is also a minimal model of $T''$. We may take blow-ups $\eta: Y'' \longrightarrow Y$ of $Y$ such that $Y'' \longrightarrow \Gamma$ is a morphism and $T''$ is a generic irreducible element of Mov$|m_0K_{Y''}|$. So we may work on $Y''$ instead. Without loss of generality, we may and do assume from the very beginning that $Y'$ (and $T'$ respectively) satisfies all the properties of $Y''$ (and $T''$ respectively) and $\pi_T = \vartheta$.

Set $|N|=\text{Mov}|t_1K_{T'}|$ and let $\varphi_{t_1, T}$ be the $t_1$-canonical map: $T \dashrightarrow \mathbb{P}^{P_{t_1}(T)-1}$.
Similar to the 4-fold case, take the Stein factorization of the composition: $$\varphi_{t_1, T} \circ \pi_T : T' \xlongrightarrow j \Gamma' \longrightarrow \overline{\varphi_{t_1, T}(T)}.$$ Denote by $j$ the induced projective morphism with connected fibers from $\varphi_{t_1, T} \circ \pi_T$ by Stein factorization. 

Set
$$a_{t_1,T}=
\begin{cases}
c, & \text{if $\dim(\Gamma')=1$}\ ;\\
1, &\text{if $\dim(\Gamma') \geq 2$}\ ,
\end{cases}$$
where $c=\deg{j_*\mathcal{O}_{T'}(N)}\geq P_{t_1}(T)-1$.
Let $S$ be a generic irreducible element of $|N|$. Then $S$ is nonsingular and we have
\begin{align}\label{eq: movable mK_T}
t_1\pi_T^*(K_T) \equiv a_{t_1,T}S+E_{N},
\end{align}
where $E_{N}$ is an effective $\mathbb{Q}$-divisor. Denote by $\sigma: S \longrightarrow S_0$ the contraction morphism of $S$ onto its minimal model $S_0$.

For a projective variety $X$ of general type with at worst $\mathbb{Q}$-factorial terminal singularities, define the {\it pluricanonical section index of X} to be
$$\delta(X):= min\{m | m\in \mathbb{Z}_{>0}, P_m(X) \geq 2\},$$
which is a birational invariant.

\subsection{Technical preparation}\label{tech pre}

\

We will use the following theorem which is a special form of Kawamata's extension theorem (see \cite[Theorem A]{KawaE}).

\begin{thm}(cf. \cite[Theorem 2.2]{CHP}) \label{KaE}
Let $Z$ be a nonsingular projective variety on which $D$ is a smooth divisor.  Assume that $K_Z+D \sim_{\mathbb{Q}} A+B$ where $A$ is an ample $\bQ$-divisor and $B$ is an effective $\bQ$-divisor such that $D\not\subseteq \textrm{Supp}(B)$.
Then the natural homomorphism
$$H^0(Z, m(K_Z+D))\longrightarrow H^0(D, mK_D)$$
is surjective for any integer $m>1$.
\end{thm}

In particular, when $Z$ is of general type and $D$,  as a generic irreducible element, moves in a base point free linear system, the conditions of Theorem \ref{KaE} are automatically satisfied. Keep the settings as in \ref{setup} and \ref{notations}. Take $Z=Y'$ and $D=T'$.

If $|M_{m_0}|$ (resp. $\mathrm{Mov}|t_1K_{T'}|$) is composed of an irrational pencil, by \cite[Lemma 2.5]{Chen10}, we have
\begin{align}\label{eq: restr T}
\pi^*(K_Y)|_{T'}=\pi_{T}^*(K_T)
\end{align}
(resp.
\begin{align}\label{eq: restr S}
\pi_T^*(K_T)|_S=\sigma^*(K_{S_0}).
\end{align}
).

If $|M_{m_0}|$ is not composed of an irrational pencil, we have $M_{m_0}\ge \theta_{m_0} T'$. For a sufficiently large and divisible integer $n>0$, we have
$$|n(\frac{m_0}{\theta_{m_0}}+1)K_{Y'}| \lsgeq |n(K_{Y'}+T')|$$
and the natural restriction map
$$H^0(Y', n(K_{Y'}+T')) \longrightarrow H^0(T', n K_{T'})$$
is surjective. By \cite[Theorem 3.3]{K-M}, $|nK_T|$ is base point free, which implies that $\textrm{Mov}|n K_{T'}|=|n\pi_T^*(K_T)|$.
We deduce that
$$n(\frac{m_0}{\theta_{m_0}}+1)\pi^*(K_Y)|_{T'} \geq M_{n(\frac{m_0}{\theta_{m_0}}+1)}|_{T'} \geq  n\pi_T^*(K_T).$$
Together with \eqref{eq: restr T}, we deduce the {\it canonical restriction inequality}:
\begin{equation}
\pi^*(K_Y)|_{T'}   \geq \frac{\theta_{m_0}}{m_0+\theta_{m_0}}\pi_T^*(K_T).\label{cri}
\end{equation}
Similarly, one has
\begin{equation}
\pi_T^*(K_T)|_S \geq \frac{a_{t_1,T}}{t_1+a_{t_1,T}}\sigma^*(K_{S_0}). \label{cri1}
\end{equation}

We will tacitly use the birationality principle of the following type.

\begin{thm} (cf. \cite[2.7]{EXP2}) \label{bir. prin.}
Let $Z$ be a nonsingular projective variety, $A$ and $B$ be two divisors on $Z$ with $|A|$ being a base point free linear system. Take the Stein factorization of $\varphi_{|A|}: Z \overset{h} \longrightarrow W \longrightarrow \bP^{h^0(Z,A)-1}$, where $h$ is a fibration onto a normal variety $W$. Then the rational map $\varphi_{|B+A|}$ is birational onto its image if one of the following conditions is satisfied:
\begin{itemize}
\item [(i)] $\dim \varphi_{|A|}(Z)\geq 2$, $|B|\neq \emptyset$ and $\varphi_{|B+A|}|_D$ is birational for a
general member $D$ of $|A|$.

\item [(ii)] $\dim \varphi_{|A|}(Z)=1$, $\varphi_{|B+A|}$ can distinguish different general fibers of $h$ and $\varphi_{|B+A|}|_F$ is birational for a general fiber $F$ of $h$.
\end{itemize}
\end{thm}

\subsection{Some useful lemmas}\label{known results}

\

The following results on surfaces and 3-folds will be used in our proof.

\begin{lem}\label{lem: adjoint linear system surface}
Let $S$ be a smooth surface of general type. Denote by $\sigma: S\to S_0$ the contraction morphism onto its minimal model. Let $Q$ be a nef and big $\mathbb{Q}$-divisor on $S$. Then $\varphi_{|K_S+3\sigma^*(K_{S_0})+\roundup{Q}|}$ is birational.
\end{lem}
\begin{proof}
If $p_g(S)>0$, $\varphi_{|K_S+3\sigma^*(K_{S_0})+\roundup{Q}|}$ is birational by \cite[Theorem 3.2(2)]{Chen10}.

If $p_g(S)=0$, by \cite[Lemma 2.5]{EXP3}, we have $(\sigma^*(K_{S_0})\cdot C)\ge 2$ for any irreducible curve $C$ passing through a very general point of $S$.  So $((3\sigma^*(K_{S_0})+Q)\cdot C)> 6$. Note that $(3\sigma^*(K_{S_0})+Q)^2>9K_{S_0}^2>8$. By \cite[Lemma 2.5]{C14}, $\varphi_{|K_S+3\sigma^*(K_{S_0})+\roundup{Q}|}$ is birational.
	
\end{proof}

\begin{lem}(\cite[Lemma 2.14]{EXP2})\label{eff. for surf.}
Let $S$ be a nonsingular projective surface of general type. Denote by $\sigma: S \longrightarrow S_0$ the blow-down onto its minimal model $S_0$. Let $Q$ be a $\mathbb{Q}$-divisor on $S$. Then $h^0(S, K_S+\roundup{Q}) \geq 2$ under one of the following conditions:
\begin{itemize}
\item[(1)] $p_g(S)>0, Q \equiv \sigma^*(K_{S_0})+Q_1$ for some nef and big $\mathbb{Q}$-divisor $Q_1$ on $S$;
\item[(2)] $p_g(S)=0, Q \equiv 2\sigma^*(K_{S_0})+Q_2$ for some nef and big $\mathbb{Q}$-divisor $Q_2$ on $S$.
\end{itemize}
\end{lem}

The following lemma is just a slight modification of \cite[Proposition 3.3]{Chen10}.

\begin{lem}\label{lem: nonvanishing adjoint linear system on 3fold}
Let $T$ be a minimal 3-fold of general type with $P_{t_1}(T)\ge 2$. Take a birational modification $\pi_T: T'\to T$ as in \ref{notations}. Keep the same notation and setting as in \ref{notations}. Let $S$ be a generic irreducible element of $|N|=\mathrm{Mov}|t_1K_{T'}|$, which is base point free by our assumption.  Suppose that $Q_\lambda\equiv\lambda \pi_T^*(K_T)$ is a nef $\mathbb{Q}$-divisor on $T'$. If $\lambda> 3t_1+2$, we have $h^0(T', K_{T'}+\roundup{Q_\lambda})\ge 2$. Moreover, if $p_g(S)>0$, we have $h^0(T', K_{T'}+\roundup{Q_\lambda})\ge 2$ for all $\lambda> 2t_1+1$.
\end{lem}
\begin{proof}
Without loss of generality(by the same argument as in the proof of \cite[Lemma 2.8]{Chen10}),  we may assume that $\mathrm{Supp}(Q_\lambda)\cup \mathrm{Supp}(E_{\pi_T}+E_N)$ has only simple normal crossings, where $\mathrm{Supp}(E_{\pi_T})$ is the support of all $\pi_T$-exceptional divisors. Note that we always have $\lambda>2t_1+1$ by our assumption. By \eqref{eq: movable mK_T},
$$Q_\lambda-\frac{1}{a_{t_1,T}}E_N-S\equiv (\lambda-\frac{t_1}{a_{t_1, T}})\pi_T^*(K_T)$$
is nef and big and it has simple normal crossing support by assumption. By Kawamata-Viehweg vanishing theorem(\cite{KV,VV}), we have
\begin{align}\label{eq: lem nonv rest1}
|K_{T'}+\roundup{Q_\lambda}||_S\lsgeq & |K_{T'}+\roundup{Q_\lambda-\frac{1}{a_{t_1,T}}E_N}||_S\nonumber\\
\lsgeq & |K_S+\roundup{(Q_\lambda-\frac{1}{a_{t_1,T}}E_N-S)|_S}|.
\end{align}
By \eqref{cri1}, we may write $\pi_T^*(K_T)|_S\simQ \frac{a_{t_1,T}}{t_1+a_{t_1,T}}\sigma^*(K_{S_0})+H$ for an effective $\mathbb{Q}$-divisor $H$. Thus we have
	$$
	(Q_\lambda-\frac{1}{a_{t_1,T}}E_N-S)|_S-(\lambda-\frac{t_1}{a_{t_1,T}})H\equiv (\lambda-\frac{t_1}{a_{t_1,T}})\cdot\frac{a_{t_1,T}}{t_1+a_{t_1,T}}\sigma^*(K_{S_0}).
	$$
	Note that $(\lambda-\frac{t_1}{a_{t_1,T}})\cdot\frac{a_{t_1,T}}{t_1+a_{t_1,T}}\ge \frac{\lambda-t_1}{t_1+1}$. If $\lambda>3t_1+2$ (resp. $\lambda>2t_1+1$ and $p_g(S)>0$), we have $(\lambda-\frac{t_1}{a_{t_1,T}})\cdot\frac{a_{t_1,T}}{t_1+a_{t_1,T}}>2$ (resp. $(\lambda-\frac{t_1}{a_{t_1,T}})\cdot\frac{a_{t_1,T}}{t_1+a_{t_1,T}}>1$ and $p_g(S)>0$). By Lemma \ref{eff. for surf.}, if $\lambda>3t_1+2$ (resp. $\lambda>2t_1+1$ and $p_g(S)>0$), we have $h^0(S, K_S+\roundup{(Q_\lambda-\frac{1}{a_{t_1,T}}E_N-S)|_S-(\lambda-\frac{t_1}{a_{t_1,T}})H})\ge 2$, which implies that $h^0(S, \\ K_S+\roundup{(Q_\lambda-\frac{1}{a_{t_1,T}}E_N-S)|_S})\ge 2$ (resp. the same inequality holds when $p_g(S)>0$ and $\lambda>2t_1+1$). We deduce that $h^0(T', K_{T'}+\roundup{Q_\lambda})\ge 2$ by \eqref{eq: lem nonv rest1}.
\end{proof}
\begin{rem}\label{rem: nonvanishing adjoint linear system on 3fold}
	 For any minimal $3$-fold $T$ of general type, by \cite[3.6]{Chen10}, one of the following holds:
	 \begin{itemize}
	 	\item[(a)] $\chi(\mathcal{O}_T)>1$ and $q(T)=0$. In this case,  we have $\delta(T)\le 18$ and $p_g(S)>0$;
	 	\item[(b)] $\delta(T)\le 10$.
	 \end{itemize}
	By Lemma \ref{lem: nonvanishing adjoint linear system on 3fold}, we always have $h^0(T', K_{T'}+\roundup{Q_\lambda})\ge 2$ if $\lambda>37$.
\end{rem}

By adopting the idea in the proof of \cite[Theorem 3.5]{Chen10}, we have the following lemma.
\begin{lem}\label{lem: biration adjoint linear system on 3fold}
	Keep the same notation as in Lemma \ref{lem: nonvanishing adjoint linear system on 3fold}. If $\lambda>4t_1+3$, $\varphi_{|K_{T'}+\roundup{Q_\lambda}|}$ is birational onto its image. In particular, for any $\lambda>75$, $\varphi_{|K_{T'}+\roundup{Q_\lambda}|}$ is birational onto its image.
\end{lem}
\begin{proof}
	By \cite[Lemma 2.8]{Chen10}, we can assume that $\mathrm{Supp}(Q_\lambda)\cup \mathrm{Supp}(E_{\pi_T}+E_N)$ has only simple normal crossings, where $\mathrm{Supp}(E_{\pi_T})$ is the support of all $\pi_T$-exceptional divisors.
	
	By  Lemma \ref{lem: nonvanishing adjoint linear system on 3fold}, if $\lambda>4t_1+2$, we have
	$$
	h^0(T', K_{T'}+\roundup{Q_\lambda-E_N-a_{t_1,T}S})\ge 2.
	$$
	By \cite[Lemma 2.7]{Chen10}(take $V=T'$, $R=Q_\lambda-E_N-a_{t_1,T}S$, $L=\mathrm{Mov}|t_1K_{T'}|$), $|K_{T'}+\roundup{Q_\lambda-E_N}|$ separates different generic irreducible elements in $\mathrm{Mov}|t_1K_{T'}|$. Thus $|K_{T'}+\roundup{Q_\lambda}|$ separates different generic irreducible elements in $\mathrm{Mov}|t_1K_{T'}|$.
	
	By Theorem \ref{bir. prin.}, we only need to show that $\varphi_{|K_{T'}+\roundup{Q_\lambda}|}|_S$ is birational. By the same argument as in the proof of Lemma \ref{lem: nonvanishing adjoint linear system on 3fold}, we have
	\begin{align*}
	|K_{T'}+\roundup{Q_\lambda}||_S\lsgeq&|K_S+\roundup{(Q_\lambda-\frac{1}{a_{t_1,T}}E_N-S)|_S}| \\
	\lsgeq& |K_S+\roundup{(Q_\lambda-\frac{1}{a_{t_1,T}}E_N-S)|_S-(\lambda-\frac{t_1}{a_{t_1,T}})H}|,
	\end{align*}
	where $H$ is the same effective $\mathbb{Q}$-divisor as in the proof of Lemma \ref{lem: nonvanishing adjoint linear system on 3fold}.
	Note that we have
	$$
	(Q_\lambda-\frac{1}{a_{t_1,T}}E_N-S)|_S-(\lambda-\frac{t_1}{a_{t_1,T}})H\equiv (\lambda-\frac{t_1}{a_{t_1,T}})\cdot\frac{a_{t_1,T}}{t_1+a_{t_1,T}}\sigma^*(K_{S_0})
	$$
	and
	$$
	(\lambda-\frac{t_1}{a_{t_1,T}})\cdot\frac{a_{t_1,T}}{t_1+a_{t_1,T}}\ge \frac{\lambda-t_1}{t_1+1}>3,
	$$
	where the last inequality holds by our assumption $\lambda>4t_1+3$. Thus
	$\varphi_{|K_S+\roundup{(Q_\lambda-\frac{1}{a_{t_1,T}}E_N-S)|_S-(\lambda-\frac{t_1}{a_{t_1,T}})H}|}$ is birational by Lemma \ref{lem: adjoint linear system surface}. We deduce that $\varphi_{|K_{T'}+\roundup{Q_\lambda}|}$ is birational if $\lambda>4t_1+3$.  By Remark \ref{rem: nonvanishing adjoint linear system on 3fold}, we may take $t_1\le 18$. Thus $\varphi_{|K_{T'}+\roundup{Q_\lambda}|}$ is birational onto its image if $\lambda>75$.
	
\end{proof}

\section{\bf Proof of the main theorem}

As an overall discussion, we keep the notation and settings in \ref{setup} and \ref{notations}.

Let $m\ge 38m_0+39$ be a positive integer. Since
$$
(m-1)\pi^*(K_Y)-\frac{1}{\theta_{m_0}}E_{m_0}-T'\equiv (m-1-\frac{m_0}{\theta_{m_0}})\pi^*(K_Y)
$$
is nef and big and it has only simple normal crossings, by Kawamata-Viehweg vanishing theorem and \eqref{eq: adjunction Y}, we have
\begin{align}\label{eq: restriction on T}
|mK_{Y'}||_{T'}\lsgeq& |K_{Y'}+\roundup{(m-1)\pi^*(K_Y)-\frac{1}{\theta_{m_0}}E_{m_0}}||_{T'}\nonumber\\
\lsgeq& |K_{T'}+\roundup{((m-1)\pi^*(K_Y)-\frac{1}{\theta_{m_0}}E_{m_0}-T')|_{T'}}|
\end{align}
By \eqref{cri}, we may write
$$
\pi^*(K_Y)|_{T'}\equiv \frac{\theta_{m_0}}{m_0+\theta_{m_0}}\pi_T^*(K_T)+G,
$$
where $G$ is an effective $\mathbb{Q}$-divisor. By \eqref{eq: restriction on T}, we have
\begin{equation}\label{eq: restriction on T 2}
\resizebox{\linewidth}{!}{$|mK_{Y'}||_{T'}\lsgeq |K_{T'}+\roundup{((m-1)\pi^*(K_Y)-\frac{1}{\theta_{m_0}}E_{m_0}-T')|_{T'}-(m-1-\frac{m_0}{\theta_{m_0}})G}|.$}
\end{equation}
Note that we have
\begin{align*}
&((m-1)\pi^*(K_Y)-\frac{1}{\theta_{m_0}}E_{m_0}-T')|_{T'}-(m-1-\frac{m_0}{\theta_{m_0}})G\\
\equiv& (m-1-\frac{m_0}{\theta_{m_0}})\pi^*(K_Y)|_{T'}-(m-1-\frac{m_0}{\theta_{m_0}})G\\
\equiv &(m-1-\frac{m_0}{\theta_{m_0}})\cdot\frac{\theta_{m_0}}{m_0+\theta_{m_0}}\pi_T^*(K_T).
\end{align*}
Since
$$
(m-1-\frac{m_0}{\theta_{m_0}})\cdot\frac{\theta_{m_0}}{m_0+\theta_{m_0}}\ge \frac{m-1-m_0}{m_0+1},
$$
we have $(m-1-\frac{m_0}{\theta_{m_0}})\cdot\frac{\theta_{m_0}}{m_0+\theta_{m_0}}>37$ if $m>38(m_0+1)$, and $(m-1-\frac{m_0}{\theta_{m_0}})\cdot\frac{\theta_{m_0}}{m_0+\theta_{m_0}}>75$ if $m>76(m_0+1)$.

By Remark \ref{rem: nonvanishing adjoint linear system on 3fold} and \eqref{eq: restriction on T 2}, we have $P_m(Y)\ge 2$ for all $m\ge 38m_0+39$. This proves (1).

For (2), by Lemma \ref{lem: biration adjoint linear system on 3fold} and \eqref{eq: restriction on T 2}, $\varphi_{|mK_{Y'}|}|_{T'}$ is birational for all $m\ge 76m_0+77$. By Theorem \ref{bir. prin.}, we only need to show that $|mK_{Y'}|$ separates different generic irreducible elements of $|M_{m_0}|$. By the same argument as in the proof of (1), we have
$$
h^0(Y', K_{Y'}+\roundup{(m-1)\pi^*(K_Y)-E_{m_0}-\theta_{m_0}T'})\ge 2
$$
for all $m \!\ge\! 38m_0+39$.\! By \!\cite[Lemma 2.7]{Chen10}, $|K_{Y'}+\roundup{(m-1)\pi^*(K_Y)\!-\!
E_{m_0}}|$ separates different generic irreducible elements of $|M_{m_0}|$ for all $m\ge 38m_0+39$ (set $V=Y'$, $R=(m-1)\pi^*(K_Y)-E_{m_0}-\theta_{m_0}T'$ and $L=M_{m_0}$). We conclude that $\varphi_{|mK_Y|}$ is birational for all $m\ge 76m_0+77$. This proves (2).

For (3), we have $\pi^*(K_Y)|_{T'}\ge\frac{1}{1+m_0}\pi_T^*(K_T)$ by \eqref{cri}. We deduce that
\begin{align*}
K_Y^4\ge \frac{1}{m_0}(\pi^*(K_Y)|_{T'})^3
\ge \frac{1}{m_0(m_0+1)^3}K_T^3
\ge  \frac{1}{1680m_0(m_0+1)^3},
\end{align*}
where the first inequality follows by \eqref{eq: mov mK} and the last inequality follows by \cite[Theorem 1.6]{EXP3}. The proof is completed.

{\bf Acknowledgment}. 
The author would like to express great gratitude to Professor Meng Chen for his suggestion of this problem and his guidance over this paper. The author is also obliged to Yong Hu for the enlightening discussions of the problem and his help in making the paper much more readable.

The author is supported by the Fundamental Research Funds for the Central Universities (Grant No. 2205012).



\begin{thebibliography}{999999}

\bibitem{BCHM}  C.  Birkar, P. Cascini, C. D. Hacon and J. McKernan, {\em Existence of minimal models for varieties of log general type}, J. Amer. Math. Soc. {\bf 23} (2010), no. 2, 405--468.  %

\bibitem{Bom} E. Bombieri, {\em Canonical models of surfaces of general type}, Inst. Hautes \'Etudes Sci. Publ. Math. {\bf 42} (1973), 171--219. %

\bibitem{EXP1} J. A. Chen and M. Chen, {\em Explicit birational geometry of threefolds of general type, I}, Ann. Sci. \'Ec. Norm. Sup\'er. {\bf 43} (2010), 365--394.

\bibitem{EXP2} J. A. Chen and M. Chen, {\em Explicit birational geometry of threefolds of general type, II}, J. Differ. Geom. {\bf 86} (2010), 237--271.%

\bibitem{EXP3} J. A. Chen and M. Chen, {\em Explicit birational geometry for 3-folds and 4-folds of general type, III}, Compos. Math. {\bf 151} (2015), 1041--1082.%

\bibitem{Q-div} M. Chen, {\em On the $\mathbb{Q}$-divisor method and its application}, J. Pure Appl. Algebra {\bf 191} (2004), 143--156.%

\bibitem{MA} M. Chen, {\em A sharp lower bound for the canonical volume of 3-folds of general type}, Math. Ann. {\bf 337} (2007), 887--908.%

\bibitem{Chen10} M. Chen, {\em On pluricanonical systems of algebraic varieties of general type}, in {\em Algebraic geometry in East Asia-Seoul 2008} Advanced Studies in Pure Mathematics, vol. 60 (Mathematical Society of Japan, Tokyo, 2010), 215-236.%

\bibitem{C14} M. Chen, {\em Some birationality criteria on 3-folds with $p_g>1$}, Sci. China Math.  {\bf 57}(2014), 2215--2234.%

\bibitem{Delta18} M. Chen, {\em On minimal 3-folds of general type with maximal pluricanonical section index}, Asian J. Math. {\bf 22} (2018), No. 2, 257-268. %

\bibitem{CJL} M. Chen, C. Jiang and B. Li, {\em On minimal varieties growing from quasismooth weighted hypersurfaces}. arXiv:2005.09828.%

\bibitem{CHP} M. Chen, Y. Hu and M. Penegini, {\em On projective threefolds of general type with small positive geometric genus}, Electron. Res. Arch. {\bf 29}(2021), no. 3, 2293--2323.%

\bibitem{H-M} C. D. Hacon and J. McKernan, {\em Boundedness of pluricanonical maps of varieties of general type}, Invent. Math. {\bf 166} (2006), 1--25. %

\bibitem{IF} A.R. Iano-Fletcher, {\em Working with weighted complete intersections}, in: Explicit birational geometry of 3-folds, London Mathematical Society Lecture Note Series, vol. 281 (Cambridge University Press, Cambridge, 2000), 101--173. %

\bibitem{KV} Y. Kawamata, {\em A generalization of Kodaira-Ramanujam's vanishing theorem}, Math. Ann. {\bf 261}(1982), 43-46. %

\bibitem{KawaE} Y. Kawamata,  {\em On the extension problem of pluricanonical forms}, in Algebraic geometry: Hirzebruch 70 (Warsaw, 1998), pp. 193--207, Contemp. Math. {\bf 241}, Amer. Math. Soc., Providence, RI, 1999. %

\bibitem{KMM}   Y. Kawamata, K. Matsuda and K. Matsuki, {\em Introduction to the minimal model problem}, in Algebraic geometry, Sendai, 1985, pp. 283--360, Adv. Stud. Pure Math. {\bf 10}, North-Holland, Amsterdam, 1987.  %

\bibitem{Kob} M. Kobayashi, {\em On Noether's inequality for threefolds}, J. Math. Soc. Japan {\bf 44} (1992), no.1, 145-156.

\bibitem{Kol} J. Koll\'ar, {\em Higher direct images of dualizing sheaves I}, Ann. Math. {\bf 123}(1986), no.1, 11-42.%

\bibitem{K-M} J. Koll\'ar and S. Mori, {\em Birational geometry of algebraic varieties}, Cambridge Tracts in Mathematics, {\bf 134}, Cambridge University Press, Cambridge, 1998. viii+254 pp. %

\bibitem{Siu}  Y. T. Siu, {\em Finite generation of canonical ring by analytic method}, Sci. China Ser. A {\bf 51} (2008), no. 4, 481--502.  %

\bibitem{Ta} S. Takayama, {\em Pluricanonical systems on algebraic varieties of general type}, Invent. Math. {\bf 165} (2006), 551--587. %

\bibitem{Tsuji}  H. Tsuji, {\em Pluricanonical systems of projective varieties of general type. I}, Osaka J. Math. {\bf 43} (2006), no. 4, 967--995. %

\bibitem{VV} E. Viehweg, {\em Vanishing theorems}, J. Reine Angew. Math. {\bf 335} (1982), 1-8. %

\bibitem{Y} J. Yan, {\em On minimal 4-folds of general type with $p_g \geq 2$}, Electron. Res. Arch. {\bf 29}(2021), no. 5, 3309--3321.%

\end{thebibliography}
\end{document}